\documentclass[oneside,english]{amsart}
\usepackage[T1]{fontenc}
\usepackage{color,enumitem,graphicx}
\usepackage[latin9]{inputenc}
\usepackage{amsthm}
\usepackage{amstext}
\usepackage{amssymb}
\usepackage{esint}

\makeatletter

\numberwithin{equation}{section}
\numberwithin{figure}{section}
\theoremstyle{plain}
\newtheorem{thm}{\protect\theoremname}[section]
  \theoremstyle{plain}
  
  \theoremstyle{plain}
  
  \theoremstyle{plain}
  
  \theoremstyle{plain}
  \newtheorem{lem}[thm]{\protect\lemmaname}
  \theoremstyle{definition}

\makeatother

\usepackage{babel}
  \providecommand{\corollaryname}{Corollary}
  \providecommand{\definitionname}{Definition}
  \providecommand{\lemmaname}{Lemma}
  \providecommand{\propositionname}{Proposition}
  \providecommand{\examplename}{Example}
\providecommand{\theoremname}{Theorem}

\DeclareMathOperator{\dv}{div}

\DeclareMathOperator{\loc}{loc}

\DeclareMathOperator{\esssup}{esssup}

\begin{document}

\title{On $(p,q)$-eigenvalues of subelliptic operators on nilpotent Lie groups}

\author{Prashanta Garain and Alexander Ukhlov}

\begin{abstract}
In the article we study the Dirichlet $(p,q)$-eigenvalue problem for subelliptic non-commutative operators on nilpotent Lie groups. We prove solvability of this eigenvalue problem and existence of the minimizer of the corresponding variational problem. 
\end{abstract}

\maketitle


\section{Introduction }
In this article we consider the Dirichlet $(p,q)$-eigenvalue problem, $1<p<\nu$, $1<q<\nu^{\ast}=\nu p (\nu-p)$ for subelliptic non-commutative operators
\begin{equation}
\label{eq1}
-\textrm{div}_{\textrm{H}}\left(|\nabla_{\textrm{H}} u|^{p-2}\nabla_{\textrm{H}} u\right)=\lambda \|u\|^{p-q}_{L^q(\Omega)}|u|^{q-2}u\,\,
\text{in}\,\,\Omega,\,\, u=0\,\,\text{on}\,\,\partial\Omega,
\end{equation}
where $\Omega$ is a bounded domain on a stratified nilpotent Lie group $\mathbb G$.

In the commutative case $\mathbb G=\mathbb R^n$ the eigenvalue problem ($p=q=2$) arises to works by Lord Rayleigh \cite{R94} where the author established the variational formulation of this problem which is based on the Dirichlet integral
$$
\|u \|^2_{W_0^{1,2}(\Omega)}=\int\limits_{\Omega} |\nabla u|^2~dx.
$$

We note classical works \cite{PW61, PS51} devoted to eigenvalues of linear elliptic operators and its connections with problems of continuum mechanics. 

The non-linear commutative case $p = q\ne 2$ was investigated by many authors as a typical non-linear eigenvalue problem, see, for example, \cite{AA87}, for extensive references we refer to \cite{L08}. The case $p\ne q$ was considered in \cite{AA87,DKN,FL10,O88}.

Subelliptic eigenvalue problems were considered first in \cite{FPh81}. In the recent time eigenvalue problems for $p$-sub-Laplace operators were intensively studied, for example, in \cite{CC21,M13,WNL}.

In the present work we consider the Dirichlet $(p,q)$-eigenvalue problem (\ref{eq1}) in the weak formulation: a function $u$ solves the eigenvalue problem iff $u\in W^{1,p}_0(\Omega)$ and
\begin{equation}
\label{eq2}
\int\limits_{\Omega}|\nabla_{\textrm{H}} u|^{p-2}\nabla_{\textrm{H}}u\nabla_{\textrm{H}}v~dx=\lambda \|u\|^{p-q}_{L^q(\Omega)}\int\limits_{\Omega}|u|^{q-2}uv~dx
\end{equation}
for all $v\in W^{1,p}_0(\Omega)$.

In this case we refer to $\lambda$ as an eigenvalue and $u$ as the corresponding eigenfunction. 

We prove solvability of the Dirichlet $(p,q)$-eigenvalue problem \eqref{eq1}, see Theorem \ref{newthm} and Theorem \ref{subopthm1}. Indeed, in Theorem \ref{subopthm1}, we have considered the following minimizing problem given by
$$
\lambda=\inf\limits_{u\in W^{1,p}_0(\Omega):\,\|u\|_{L^q(\Omega)}=1} \int\limits_{\Omega}|\nabla_{\textrm{H}} u|^{p}~dx
$$
and proved existence of a function $v\in W^{1,p}_0(\Omega)$, $\|v\|_{L^q(\Omega)}=1$, such that
$$
\lambda=\int\limits_{\Omega}|\nabla_{\textrm{H}} v|^{p}~dx.
$$
Moreover, we observe that $v$ is an eigenfunction corresponding to $\lambda$ and its associated eigenfunctions are precisely the scalar multiple of those vectors at which $\lambda$ is reached. Finally, in Theorem \ref{regthm}, we establish some qualitative properties of the eigenfunctions of \eqref{eq1}.

\section{Nilpotent Lie groups and Sobolev spaces}

Recall that a stratified homogeneous group \cite{FS}, or, in another
terminology, a Carnot group \cite{Pa} is a~connected simply
connected nilpotent Lie group~ $\mathbb G$ whose Lie algebra~ $V$ is
decomposed into the direct sum~ $V_1\oplus\cdots\oplus V_m$ of
vector spaces such that $\dim V_1\geqslant 2$, $[V_1,\ V_i]=V_{i+1}$
for $1\leqslant i\leqslant m-1$ and $[V_1,\ V_m]=\{0\}$. Let
$X_{11},\dots,X_{1n_1}$ be left-invariant basis vector fields of
$V_1$. Since they generate $V$, for each $i$, $1<i\leqslant m$, one
can choose a basis $X_{ik}$ in $V_i$, $1\leqslant k\leqslant
n_i=\dim V_i$, consisting of commutators of order $i-1$ of fields
$X_{1k}\in V_1$. We identify elements $g$ of $\mathbb G$ with
vectors $x\in\mathbb R^N$, $N=\sum_{i=1}^{m}n_i$, $x=(x_{ik})$,
$1\leqslant i\leqslant m$, $1\leqslant k\leqslant n_i$ by means of
exponential map $\exp(\sum x_{ik}X_{ik})=g$. Dilations $\delta_t$
defined by the formula 
\begin{multline}
\nonumber
\delta_t x= (t^ix_{ik})_{1\leqslant i\leqslant m,\,1\leqslant k\leqslant n_j}\\
=(tx_{11},...,tx_{1n_1},t^2x_{21},...,t^2x_{2n_2},...,t^mx_{m1},...,t^mx_{mn_m}),
\end{multline}
are automorphisms of
$\mathbb G$ for each $t>0$. Lebesgue measure $dx$ on $\mathbb R^N$
is the bi-invariant Haar measure on~ $\mathbb G$ (which is generated
by the Lebesgue measure by means of the exponential map), and
$d(\delta_t x)=t^{\nu}~dx$, where the number
$\nu=\sum_{i=1}^{m}in_i$ is called the homogeneous dimension of the
group~$\mathbb G$. The measure $|E|$ of a measurable subset
$E$ of $\mathbb G$ is defined by
$
|E|=\int_E~dx.
$

Recall that a continuous map $\gamma: [a,b]\to\mathbb G$ is called a continuous curve on $\mathbb G$. This continuous curve is rectifiable if
$$
\sup\left\{\sum\limits_{k=1}^m|\left(\gamma(t_{k})\right)^{-1}\gamma(t_{k+1})|\right\}<\infty,
$$
where the supremum is taken over all partitions $a=t_1<t_2<...<t_m=b$ of the segment $[a,b]$.
The length $l(\gamma)$ of a rectifiable curve $\gamma:[a,b]\to\mathbb G$ can be calculated by the formula
$$
l(\gamma)=\int\limits_a^b {\left\langle \dot{\gamma}(t),\dot{\gamma}(t)\right\rangle}_0^{\frac{1}{2}}~dt=
\int\limits_a^b \left(\sum\limits_{i=1}^{n}|a_i(t)|^2\right)^{\frac{1}{2}}~dt
$$
where ${\left\langle \cdot,\cdot\right\rangle}_0$ is the inner product on $V_1$. The result of \cite{CH} implies that one can connect two arbitrary points $x,y\in \mathbb G$ by a rectifiable curve. The Carnot-Carath\'eodory distance $d(x,y)$ is the infimum of the lengths over all rectifiable curves with endpoints $x$ and $y$ in $\mathbb G$. The Hausdorff dimension of the metric space $\left(\mathbb G,d\right)$ coincides with the homogeneous dimension $\nu$ of the group $\mathbb G$.

\subsection{Sobolev spaces on Carnot groups}

Let $\mathbb G$ be a Carnot group with one-parameter dilatation
group $\delta_t$, $t>0$, and a homogeneous norm $\rho$, and let
$E$ be a measurable subset of $\mathbb G$. The Lebesgue space $L^p(E)$, $p\in [1,\infty]$, is the space of pth-power
integrable functions $f:E\to\mathbb R$ with the standard norm:
$$
\|f\|_{L^p(E)}=\biggl(\int\limits_{E}|f(x)|^p~dx\biggr)^{\frac{1}{p}},\,\,1\leq p<\infty,
$$
and $\|f\|_{L^{\infty}(E)}=\esssup_{E}|f(x)|$ for $p=\infty$. We
denote by $L^p_{\loc}(E)$ the space of functions
$f: E\to \mathbb R$ such that $f\in L^p(F)$ for each compact
subset $F$ of $E$.

Let $\Omega$ be an open set in $\mathbb G$. The (horizontal) Sobolev space
$W^{1,p}(\Omega)$, $1\leqslant p\leqslant\infty$, consists of the functions
$f:\Omega\to\mathbb R$ which are locally integrable in $\Omega$, having the weak
derivatives $X_{1i} f$ along the horizontal vector fields $X_{1i}$, $i=1,\dots,n_1$,
and the finite norm
$$
\|f\|_{W^{1,p}(\Omega)}=\|f\|_{L^p(\Omega)}+\|\nabla_{\textrm{H}} f\|_{L^p(\Omega)},
$$
where $\nabla_\textrm{H} f=(X_{11}f,\dots,X_{1n_1}f)$ is the horizontal subgradient of $f$.
If $f\in W^{1,p}(U)$ for each bounded open set $U$ such that
$\overline{U}\subset\Omega$ then we say that $f$ belongs to the
class $W^{1,p}_{\loc}(\Omega)$.

The Sobolev space $W^{1,p}_0(\Omega)$ is defined to be the closure of $C^{\infty}_c(\Omega)$ under the norm
$$
\|f\|_{W^{1,p}_0(\Omega)}=\|f\|_{L^p(\Omega)} +\|\nabla_{\textrm{H}} f\|_{L^p(\Omega)}.
$$
For the following result, refer to \cite{F75,Vodop, Vodop1, X90}. 
\begin{lem}
\label{Xuthm}
The space $W_0^{1,p}(\Omega)$ is a real  separable and uniformly convex Banach space.
\end{lem}

The following embedding result follows from \cite[$(2.8)$]{Danielli} and \cite{F75}, \cite[Theorem $8.1$]{HK00}, see also \cite[Theorem $2.3$]{Garo1}. 

\begin{lem}\label{emb}
Let $\Omega\subset\mathbb G$ be a bounded domain and $1<p<\nu$. Then $W_0^{1,p}(\Omega)$ is continuously embedded in $L^q(\Omega)$ for every $1\leq q\leq \nu^{*}$ where $\nu^{\ast}={\nu p}/{(\nu-p)}$. Moreover, the embedding is compact for every $1\leq q<\nu^{*}$.
\end{lem}

Hence, in the case $1<p<\nu$ we can consider the Sobolev space $W_0^{1,p}(\Omega)$ with the norm
$$
\|f\|_{W^{1,p}_0(\Omega)}=\|\nabla_{\textrm{H}} f\|_{L^p(\Omega)}.
$$

\section{Dirichlet $(p,q)$-eigenvalue problems}

The system of basis vectors $X_1,X_2,\dots,X_n$ of the space $V_1$ (here and throughout we set $n_1=n$ and $X_{1i}=X_i$, where
$i=1,\dots,n$) satisfies the H\"ormander's hypoellipticity condition. We study the non-linear eigenvalue problem in the geometry of the vector fields  satisfying the H\"ormander's hypoellipticity condition.

Let $1<p<\nu$, $\lambda\in\mathbb R$ and consider the following subelliptic equation
\begin{equation}\label{subopeqn}
-\dv_{\textrm{H}}(|\nabla_{\textrm{H}} u|^{p-2} \nabla_{\textrm{H}} u)=\lambda\|u\|_{L^q(\Omega)}^{p-q}|u|^{q-2}u\text{ in }\Omega,\quad u=0\text{ on }\partial\Omega,
\end{equation}
where $1<q<\nu^*=\frac{\nu p}{\nu-p}$. We say that $(\lambda,u)\in \mathbb R\times W_0^{1,p}(\Omega)\setminus\{0\}$ is an eigenpair of \eqref{subopeqn} if for every $v\in W_0^{1,p}(\Omega)$, we have
\begin{equation}\label{subopwksol}
\int_{\Omega}|\nabla_{\textrm{H}}u|^{p-2}\nabla_{\textrm{H}} u \nabla_{\textrm{H}}v\,dx=\lambda\|u\|_{L^q(\Omega)}^{p-q}\int_{\Omega}|u|^{q-2}u v\,dx.
\end{equation}
Moreover, we refer to $\lambda$ as an eigenvalue and $u$ as the corresponding eigenfunction. 

Let us consider the following minimizing problem given by
\begin{equation}\label{subopmin}
\lambda:=\inf\left\{\int_{\Omega}|\nabla_{\textrm{H}}u|^p\,dx:v\in W_0^{1,p}(\Omega)\cap S{L^q(\Omega)}\right\},
\end{equation}
where
\begin{equation}\label{sl}
S{L^q(\Omega)}:=\{v\in L^q(\Omega):\|v\|_{L^q(\Omega)}=1\}.
\end{equation}

Let us define 
\begin{equation}\label{u}
(U,\|\cdot\|_U)=(W_0^{1,p}(\Omega),\|\cdot\|_{W_0^{1,p}(\Omega)})
\end{equation}

and
\begin{equation}\label{v}
(V,\|\cdot\|_V)=(L^q(\Omega),\|\cdot\|_{L^q(\Omega)})
\end{equation}
and we denote by $SV=S{L^q(\Omega)}$, where $SL^{q}(\Omega)$ is defined in \eqref{sl}. By Lemma \ref{Xuthm} and Lemma \ref{emb}, it follows that $(U,\|\cdot\|_U)$ is a uniformly convex Banach space, which is compactly embedded in the Banach space $(V,\|\cdot\|_V)$.
Define the operators
\begin{equation}\label{a}
A:U\to U^*\text{ by } \langle Av,w\rangle=\langle \dv\left(\nabla_{\textrm{H}}v|^{p-2}\nabla_{\textrm{H}}v\right), w\rangle=\int_{\Omega}|\nabla_{\textrm{H}}v|^{p-2}\nabla_{\textrm{H}}v \nabla_{\textrm{H}}w\,dx
\end{equation}
and
\begin{equation}\label{b}
B:V\to V^*\text{ by }\langle B(v),w\rangle=\int_{\Omega}|v|^{q-2}vw\,dx.
\end{equation}

Below, we prove the following properties of $A$ and $B$.
\begin{lem}\label{auxlmab}
\begin{enumerate}
\item[$(H_1)$] $A(tv)=|t|^{p-2}tA(v)\quad\forall t\in\mathbb{R}\quad \text{and}\quad\forall v\in U$,
\item[$(H_2)$] $B(tv)=|t|^{q-2}tB(v)\quad\forall t\in\mathbb{R}\quad \text{and}\quad\forall v\in V$,
\item[$(H_3)$] $\langle A(v),w\rangle\leq\|v\|_U^{p-1}\|w\|_U$ for all $v,w\in U$, where the equality holds if and only if $v=0$ or $w=0$ or $v=t w$ for some $t>0$.
\item[$(H_4)$] $\langle B(v),w\rangle\leq\|v\|_V^{q-1}\|w\|_V$ for all $v,w\in V$, where the equality holds if and only if $v=0$ or $w=0$ or $v=t w$ for some $t\geq 0$.
\item[$(H_5)$] For every $w\in V\setminus\{0\}$ there exists $u\in U\setminus\{0\}$ such that
$$
\langle A(u),v\rangle=\langle B(w),v\rangle\quad\forall\quad v\in U.
$$
\end{enumerate}
Moreover, $A$ and $B$ are continuous.
\end{lem}
By the property $(H_5)$ in Lemma \ref{auxlmab}, as in \cite[page $579$ and page $584-585$]{Ercole} for every $w_0\in SV$ there exists a sequence $\{w_n\}_{n\in\mathbb{N}}\subset U\cap SV$ such that 
\begin{equation}\label{its}
\langle A(w_{n+1},v\rangle=\mu_n\langle B(w_n),v\rangle\quad\forall\quad v\in U,
\end{equation}
where $\mu_n \geq \lambda$, with $\lambda$ given by \eqref{subopmin}.
\subsection{Main results}
The main results in this section are stated as follows:

\begin{thm}\label{newthm}
Let $1<p<\nu$ and $1<q<\nu^{*}$. Then, the sequences $\{\mu_n\}_{n\in\mathbb{N}}$ and $\{\|w_{n+1}\|_U^{p}\}_{n\in\mathbb{N}}$ given by \eqref{its} are nonincreasing and converge to the same limit $\mu$, which is bounded below by $\lambda$. Moreover, $\mu$ is an eigenvalue of \eqref{subopeqn} and there exists a subsequence $\{n_j\}_{j\in\mathbb{N}}$ such that both $\{w_{n_j}\}_{j\in\mathbb{N}}$ and $\{w_{n_{j+1}}\}_{j\in\mathbb{N}}$ converges in $U$ to the same limit $w\in U\cap SV$, which is an eigenvector corresponding to $\mu$.  
\end{thm}

\begin{thm}\label{subopthm1}
Let $1<p<\nu$ and $1<q<\nu^{*}$. Suppose $\{u_n\}_{n\in\mathbb{N}}\subset U\cap SV$ is a minimizing sequence for $\lambda$, that is $\|u_n\|_V=1$ and $\|u_n\|_U^{p}\to\lambda$. Then there exists a subsequence $\{u_{n_j}\}_{j\in\mathbb{N}}$ which converges weakly in $U$ to $u\in U\cap SV$ such that  
$$
\lambda=\int_{\Omega}|\nabla_{\textrm{H}}~u|^p\,dx.
$$
Moreover, $u$ is an eigenfunction corresponding to $\lambda$ and its associated eigenfunctions are precisely the scalar multiple of those vectors at which $\lambda$ is reached.
\end{thm}

Our final main result concerns the following qualitative properties of the eigenfunctions of \eqref{subopeqn}.
\begin{thm}\label{regthm}
Let $1<p<\nu$ and $1<q<\nu^{*}$. Assume that $\lambda>0$ is an eigenvalue of the problem \eqref{subopeqn} and $u\in U\setminus\{0\}$ is a corresponding eigenfunction. Then (i) $u\in L^\infty(\Omega)$. (ii) Moreover, if $u\in U\setminus\{0\}$ is nonnegative in $\Omega$, then $u>0$ in $\Omega$. Further, for every $\omega\Subset\Omega$ there exists a positive constant $c$ depending on $\omega$ such that $u\geq c>0$ in $\omega$.
\end{thm} 

To prove our main results above, first we prove Lemma \ref{auxlmab} below.\\
\textbf{Proof of Lemma \ref{auxlmab}:}
\begin{enumerate}
\item[$(H_1)$] Follows by the definition of $A$.

\item[$(H_2)$]  Follows by the definition of $B$.

\item[$(H_3)$] First, using Cauchy-Schwartz inequality and then by H\"older's inequality with exponents $\frac{p}{p-1}$ and $p$, we obtain
$$
\langle Av,w\rangle=\int_{\Omega}|\nabla_{\textrm{H}}v|^{p-2}\nabla_{\textrm{H}}v \nabla_{\textrm{H}}w\,dx\leq\int_{\Omega}|\nabla_{\textrm{H}}v|^{p-1}|\nabla_{\textrm{H}}w|\,dx\leq\|v\|_{U}^{p-1}\|w\|_{U}.
$$
If $v=0$ or $w=0$ then the equality $\langle Av,w\rangle=\|v\|_{U}^{p-1}\|w\|_{U}$ holds. So we assume this equality such that both $v\neq 0$ and $w\neq 0$. Then the equality in Cauchy-Schwartz and H\"older's inequality hold simultaneously. That is, at one end (due to the case of the equality in Cauchy-Schwartz inequality) we get
$$
\int_{\Omega}|\nabla_{\textrm{H}} v|^{p-2}\nabla_{\textrm{H}}v \nabla_{\textrm{H}}w\,dx=\int_{\Omega}|\nabla_{\textrm{H}}v|^{p-1}|\nabla_{\textrm{H}}w|\,dx,
$$
which gives $|\nabla_{\textrm{H}} v|^{p-2}\nabla_{\textrm{H}}v \nabla_{\textrm{H}}w=|\nabla_{\textrm{H}}v|^{p-1}|\nabla_{\textrm{H}}w|$ and hence, $\nabla_{\textrm{H}}v(x)=c(x)\nabla_{\textrm{H}}w(x)$ for almost every $x\in\Omega$ for some $c(x)\geq 0$. Also, due to the equality in H\"older's inequality, we have
$$
\int_{\Omega}|\nabla_{\textrm{H}}v|^{p-2}\nabla_{\textrm{H}}v \nabla_{\textrm{H}}w\,dx=\|v\|_{U}^{p-1}\|w\|_{U},
$$
which gives $|\nabla_{\textrm{H}}v|=t|\nabla_{\textrm{H}}w|$ in $\Omega$ for some constant $t>0$. Therefore, $c(x)=t$ in $\Omega$. Hence $\nabla_{\textrm{H}}v=t \nabla_{\textrm{H}}w$ in $\Omega$ and therefore $v=tw$ in $\Omega$ for some $t>0$. Thus, $(H_3)$ holds. 

\item[$(H_4)$] This property can be verified similarly as in $(H_3)$.

\item[$(H_5)$] We prove it by applying Theorem \ref{MB} as follows:

\vskip 0.2cm

\noindent
\textbf{Boundedness:} From H\"older's inequality, we have
$$
\|Av\|_{U^*}=\sup_{\|w\|_U\leq 1}|\langle Av,w\rangle|\leq\|v\|_U^{p-1}\|w\|_U\leq\|v\|^{p-1}_U.
$$
Thus, $A$ is bounded.

\vskip 0.2cm

\noindent
\textbf{Continuity:} Suppose $v_n\in U$ such that $v_n\to v$ in the norm of $U$. Thus, upto a subsequence $\nabla_{\textrm{H}}v_{n_j}\to \nabla_{\textrm{H}}v$ pointwise almost everywhere in $\Omega$. We observe that 
$$
\||\nabla_{\textrm{H}}v_{n_j}|^{p-2}\nabla_{\textrm{H}}v_{n_j}\|_{L^\frac{p}{p-1}(\Omega)}\leq\|\nabla_{\textrm{H}}v_{n_j}\|^{p-1}_{U}\leq C, 
$$
for some uniform constant $C>0$, which is independent of $n$. Thus, 
$$
|\nabla_{\textrm{H}}v_{n_j}|^{p-2}\nabla_{\textrm{H}}v_{n_j}\rightharpoonup |\nabla_{\textrm{H}}v|^{p-2}\nabla_{\textrm{H}}v
$$ 
weakly in $L^\frac{p}{p-1}(\Omega)$. Since, the weak limit is independent of the choice of the subsequence, it follows that
$$
|\nabla_{\textrm{H}}v_{n}|^{p-2}\nabla_{\textrm{H}}v_{n}\rightharpoonup |\nabla_{\textrm{H}}v|^{p-2}\nabla_{\textrm{H}}v
$$ 
weakly in $L^\frac{p}{p-1}(\Omega)$. As a consequence, we have 
$$
\langle Av_n,w\rangle\to\langle Av,w\rangle
$$
for every $w\in U$. Thus $A$ is continuous. Similarly, we obtain $B$ is continuous.

\noindent
\textbf{Coercivity:} We observe that 
$$
\langle Av,v\rangle=\int_{\Omega}|\nabla_{\textrm{H}}v|^p\,dx.
$$
Since $p>1$, the operator $A$ is coercive.

\noindent
\textbf{Monotonicity:} By Lemma \ref{alg} we have
$$
\langle Av-Aw,v-w\rangle=\int_{\Omega}\langle|\nabla_{\textrm{H}}v|^{p-2}\nabla_{\textrm{H}}v-|\nabla_{\textrm{H}}w|^{p-2}\nabla_{\textrm{H}}w,\nabla_{\textrm{H}}(v-w)\rangle\,dx\geq 0,
$$
for every $v,w\in U$. Thus, $A$ is monotone.

Now, by the continuous embedding of $U$ in $V$ from Lemma \ref{emb}, we observe that $B(w)\in U^*$ for every $w\in V\setminus\{0\}$. Note that by Lemma \ref{Xuthm}, it follows that $U$ is a separable and reflexive Banach space. Therefore, by Theorem \ref{MB} there exists $u\in U\setminus\{0\}$ such that
$$
\langle A(u),v\rangle=\langle B(w),v\rangle\quad\forall v\in U.
$$
Hence the hypothesis $(H_5)$ holds. This completes the proof. \qed
\end{enumerate}

The next result is useful to prove boundedness of the eigenfunctions of \eqref{subopeqn}.
\begin{lem}\label{embd}
Let $\Omega\subset\mathbb G$ be such that $|\Omega|<\infty$ and $1<p<\nu$, $1<l<\nu^{*}=\frac{\nu l}{\nu-l}$. Then for every $u\in W_0^{1,p}(\Omega)$, there exists a positive constant $C=C(p,l,\nu)$ such that
\begin{equation}\label{embeqn}
\left(\int_{\Omega}|u|^l\,dx\right)^\frac{1}{l}\leq C|\Omega|^{\frac{1}{l}-\frac{1}{p}+\frac{1}{\nu}}\left(\int_{\Omega}|\nabla_{\textrm{H}}u|^p\,dx\right)^\frac{1}{p}.
\end{equation}
\end{lem}
\begin{proof}
Proceeding as in \cite[Corollary $1.57$]{Maly}, we set 
$$
s=
\begin{cases}
1,\text{ if }l\nu\leq \nu+l\\
\frac{\nu l}{\nu+l},\text{ if }\nu l>\nu+l.
\end{cases}
$$
Then $1\leq s\leq p$, $s<\nu$ and $s^*=\frac{\nu s}{\nu-s}\geq l$. Using H\"older's inequality along with Lemma \ref{emb}, we obtain
\begin{equation}\label{embeqn3}
\begin{split}
\|u\|_{L^l(\Omega)}\leq\|u\|_{L^{s^*}(\Omega)}|\Omega|^{\frac{1}{l}-\frac{1}{s}+\frac{1}{\nu}}&\leq C\|\nabla_{\textrm{H}}u\|_{L^s(\Omega)}|\Omega|^{\frac{1}{l}-\frac{1}{s}+\frac{1}{\nu}}\\
&\leq C\|\nabla_{\textrm{H}}u\|_{L^p(\Omega)}|\Omega|^{\frac{1}{l}-\frac{1}{p}+\frac{1}{\nu}}.
\end{split}
\end{equation}
Hence the proof.
\end{proof}

\subsection{Proof of the main results:}
\textbf{Proof of Theorem \ref{newthm}:} The proof follows by Lemma \ref{auxlmab} and \cite[Theorem $1$]{Ercole}. \qed \\
\textbf{Proof of Theorem \ref{subopthm1}:} The proof follows by Lemma \ref{auxlmab} and \cite[Proposition $2$]{Ercole}. \\ 
\textbf{Proof of Theorem \ref{regthm}:}
\textbf{(i)} Due to the homogeneity of the equation \eqref{subopeqn}, without loss of generality, we assume that $\|u\|_{V}=1$. Let $k\geq 1$ and set $A(k):=\{x\in\Omega:u(x)>k\}$. Choosing $v=(u-k)^+$ as a test function in \eqref{subopwksol}, we obtain
\begin{equation}\label{regtst1}
\int_{A(k)}|\nabla_{\textrm{H}} u|^p\,dx=\lambda\int_{A(k)}|u|^{q-2}u(u-k)\,dx\leq\lambda\int_{A(k)}|u|^{q-1}(u-k)\,dx.
\end{equation}
We prove the result in the following two cases:\\
\textbf{Case $I$.} Let $q\leq p$, then since $k\geq 1$, over the set $A(k)$, we have $|u|^{q-1}\leq|u|^{p-1}$. Therefore, from \eqref{regtst1} we have
\begin{equation}\label{regtst2}
\begin{split}
\int_{A(k)}|\nabla_{\textrm{H}} u|^p\,dx&=\lambda\int_{A(k)}|u|^{q-2}u(u-k)\,dx\\
&\leq\lambda\int_{A(k)}|u|^{p-1}(u-k)\,dx\\
&\leq\lambda\int_{A(k)}(2^{p-1}(u-k)^{p}+2^{p-1}k^{p-1}(u-k))\,dx,
\end{split}
\end{equation}
where to obtain the last inequality above, we have used the inequality $(a+b)^{p-1}\leq 2^{p-1}(a^{p-1}+b^{p-1})$ for $a,b\geq 0$. Using the Sobolev inequality \eqref{embeqn} with $l=p$ in \eqref{regtst2} we obtain
\begin{equation}\label{regtst3}
\begin{split}
(1-S\lambda 2^{p-1}|A(k)|^\frac{p}{\nu})\int_{A(k)}(u-k)^p\,dx&\leq\lambda S 2^{p-1}k^{p-1}|A(k)|^\frac{p}{\nu}\int_{A(k)}(u-k)\,dx,
\end{split}
\end{equation}
where $S>0$ is the Sobolev constant. Note that $\|u\|_{L^1(\Omega)}\geq k|A(k)|$ and therefore for every $k\geq k_0=(2^p S\lambda)^\frac{\nu}{p}\|u\|_{L^1(\Omega)}$, we have $S\lambda2^{p-1}|A(k)|^\frac{p}{\nu}\leq\frac{1}{2}$. Using this fact in \eqref{regtst3}, for every
$k\geq\max\{k_0,1\}$, we get
\begin{equation}\label{regtst4}
\begin{split}
\int_{A(k)}(u-k)^p\,dx&\leq \lambda S 2^{p}k^{p-1}|A(k)|^\frac{p}{\nu}\int_{A(k)}(u-k)\,dx.
\end{split}
\end{equation}
Using H\"older's inequality and the estimate \eqref{regtst4} we obtain
\begin{equation}\label{regtst5}
\int_{A(k)}(u-k)\,dx\leq (\lambda S2^p)^\frac{1}{p-1}k|A(k)|^{1+\frac{p}{\nu(p-1)}}.
\end{equation}
Noting \eqref{regtst5}, by \cite[Lemma $5.1$]{Ural}, we get $u\in L^\infty(\Omega)$. \\
\textbf{Case $II.$} Let $q>p$, then using the inequality $(a+b)^{q-1}\leq 2^{q-1}(a^{q-1}+b^{q-1})$ for $a,b\geq 0$ in \eqref{regtst1} we get
\begin{equation}\label{regtstc2}
\int_{A(k)}|\nabla_{\textrm{H}}u|^p\,dx\leq\lambda\int_{A(k)}(2^{q-1}(u-k)^{q}+2^{q-1}k^{q-1}(u-k))\,dx.
\end{equation}
Now, using the Sobolev inequality \eqref{embeqn} with $l=q$ in the estimate \eqref{regtstc2} we obtain
\begin{equation}\label{regtstc21}
\begin{split}
\left(\int_{A(k)}(u-k)^q\,dx\right)^\frac{p}{q}&\leq S\lambda|A(k)|^{p(\frac{1}{q}-\frac{1}{p}+\frac{1}{\nu})}\int_{A(k)}(2^{q-1}(u-k)^q+2^{q-1}k^{q-1}(u-k))\,dx,
\end{split}
\end{equation}
where $S>0$ is the Sobolev constant. Since $\int_{A(k)}(u-k)^q\,dx\leq\|u\|^{q}_{L^q(\Omega)}=1$ and $q>p$, the quantity in the left side of \eqref{regtstc21} can be estimated from below as

\begin{equation}\label{regtstc22}
\left(\int_{A(k)}(u-k)^q\,dx\right)^\frac{p}{q}=\left(\int_{A(k)}(u-k)^q\,dx\right)^{\frac{p-q}{q}+1}\geq\int_{A(k)}(u-k)^q\,dx.
\end{equation}
Using \eqref{regtstc22} in \eqref{regtstc21} we get
\begin{equation}\label{regtstc23}
\begin{split}
&\Big(1-S\lambda 2^{q-1}|A(k)|^{p(\frac{1}{q}-\frac{1}{p}+\frac{1}{\nu})}\Big)\int_{A(k)}(u-k)^q\,dx\\
&\leq S\lambda2^{q-1}k^{q-1}|A(k)|^{p(\frac{1}{q}-\frac{1}{p}+\frac{1}{\nu})}\int_{A(k)}(u-k)\,dx.
\end{split}
\end{equation}
Let $\alpha={p(\frac{1}{q}-\frac{1}{p}+\frac{1}{\nu})}$, which is positive, since $1<q<\nu^{*}$. Choosing $k_1=(S\lambda2^q)^\frac{1}{\alpha}\|u\|_{L^1(\Omega)}$, due to the fact that $k|A(k)|\leq\|u\|_{L^1(\Omega)}$, we obtain for every $k\geq k_1$ that $S\lambda 2^{q-1}|A(k)|^\alpha\leq\frac{1}{2}$. Using this property in \eqref{regtstc23}, we have

\begin{equation}\label{regtstc24}
\begin{split}
\int_{A(k)}(u-k)^q\,dx&\leq S\lambda2^{q}k^{q-1}|A(k)|^\alpha\int_{A(k)}(u-k)\,dx.
\end{split}
\end{equation}
By H\"older's inequality and the estimate \eqref{regtstc24} we arrive at
\begin{equation}\label{regtstc25}
\int_{A(k)}(u-k)\,dx\leq (\lambda S2^q)^\frac{1}{q-1}k|A(k)|^{1+\frac{\alpha}{q-1}}.
\end{equation}
Noting \eqref{regtstc25}, by \cite[Lemma $5.1$]{Ural}, we get $u\in L^\infty(\Omega)$.\\
\textbf{(ii)} By \cite[Theorem $5$]{Vodop}, the result follows. \qed

\section{{Coercive operators in Banach spaces}}
In this section, we state some auxiliary results. The first one is the following algebraic inequality from \cite[Lemma $2.1$]{Dama}.

\begin{lem}\label{alg}
Let $1<p<\infty$. Then for any $a,b\in\mathbb{R}^N$, there exists a constant $C=C(p)>0$ such that
\begin{equation}\label{algineq}
\langle |a|^{p-2}a-|b|^{p-2}b, a-b \rangle\geq
C(|a|+|b|)^{p-2}|a-b|^2.
\end{equation}
\end{lem}

Next, we state the following result, which follows from \cite[Theorem $9.14$]{var}.
\begin{thm}\label{MB}
Let $V$ be a real separable reflexive Banach space and $V^*$ be the dual of $V$. Assume that $A:V\to V^{*}$ is a bounded, continuous, coercive and monotone operator. Then $A$ is surjective, i.e., given any $f\in V^{*}$, there exists $u\in V$ such that $A(u)=f$. If $A$ is strictly monotone, then $A$ is also injective. 
\end{thm}

\section*{Acknowledgement} The authors thank Professor Giovanni Franzina for some fruitful discussion on the topic.

\vskip 0.3cm

Department of Mathematics, Ben-Gurion University of the Negev, P.O.Box 653, Beer Sheva, 8410501, Israel

\emph{E-mail address:} \email{pgarain92@gmail.com}

\vskip 0.3cm

Department of Mathematics, Ben-Gurion University of the Negev, P.O.Box 653, Beer Sheva, 8410501, Israel

\emph{E-mail address:} \email{ukhlov@math.bgu.ac.il}


\begin{thebibliography}{References}

\bibitem{AA87} J.~Garc\'ia Azorero and I.~Peral Alonso, Existence and nonuniqueness for the $p$-Laplacian: nonlinear eigenvalues, Comm. Partial Differential Equations, \textbf{12} (1987), 1389--1430.

\bibitem{Garo1} L.~Capogna, D.~Danielli, N.~Garofalo, An embedding theorem and the Harnack inequality for nonlinear subelliptic equations,
Comm. Partial Differential Equations, \textbf{18} (1993), 1765--1794.

\bibitem{CC21} H.~Chen, H.~G,~Chen, Estimates of Dirichlet eigenvalues for a class of sub-elliptic operators, Proc. London Math. Soc., \textbf{122} (2021),808--847.

\bibitem{CH} W.~L.~Chow, {Systeme von linearen partiellen differential gleichungen erster ordnung}, Math. Ann.,~\textbf{117} (1939), 98--105.

\bibitem{var} Ph.~G.~Ciarlet, Linear and nonlinear functional analysis with applications, Society for Industrial and
Applied Mathematics, Philadelphia, PA, 2013.

\bibitem{Dama} L.~Damascelli, Comparison theorems for some quasilinear degenerate elliptic operators and applications
to symmetry and monotonicity results, Ann. Inst. H. Poincar\'e Anal. Non Lin\'eaire, \textbf{15} (1998), 493--516.

\bibitem{Danielli} D.~Danielli, Regularity at the boundary for solutions of nonlinear subelliptic equations, Indiana
Univ. Math. J., \textbf{44} (1995), 269--286.

\bibitem{DKN} P.~Dr\'abek, A.~Kufner, F.~Nicolosi, Quasilinear elliptic equations with degenerations and singularities, de Gruyter Series in Nonlinear Analysis and Applications, \textbf{5}, Walter de Gruyter\& Co., Berlin, (1997).

\bibitem{Ercole} G.~Ercole, Solving an abstract nonlinear eigenvalue problem by the inverse iteration method, Bull. Braz.
Math. Soc. (N.S.), \textbf{49} (2018), 577--591.


\bibitem{FPh81} C.~Fefferman, D.~H.~Phong, Subelliptic eigenvalue problems, Proceedings of the Conference on Harmonic Analysis, in honor of A. Zygmund, Wadsworth Math. Series (1981), 590--606.

\bibitem{F75} G.~B.~Folland, Subelliptic estimates and function spaces on nilpotent Lie groups, Ark.Math., \textbf{13} (1975), 161--207.

\bibitem{FS} G.~B.~Folland, E.~M.~Stein, {Hardy spaces on homogeneous group}, Princeton Univ. Press. (Princeton, 1982).

\bibitem{FL10} G.~Franzina, P.~D.~Lamberti, Electronic Journal of Differential Equations, \textbf{26} (2010), 1--10.

\bibitem{HK00} P.~Hajlasz, P.~Koskela, Sobolev met Poincare, Mem. Amer. Math. Soc., \textbf{688}, (2000).

\bibitem{L08} P.~Lindqvist, A nonlinear eigenvalue problem. Topics in mathematical analysis, Ser. Anal.
Appl. Comput., World Sci. Publ., Hackensack, NJ, \textbf{3} (2008), 175--203.

\bibitem{Maly} Jan Mal\'{y} and William~P. Ziemer, Fine regularity of solutions of elliptic partial differential equations, volume~51 of Mathematical Surveys and Monographs. American Mathematical Society, Providence, RI, 1997.

\bibitem{M13} F.~Montefalcone, Geometric inequalities in Carnot groups, Pacific J. Math., \textbf{263} 2013, 171--206.
Vol. 263, No. 1, 2013


\bibitem{Ural} Ladyzhenskaya, Olga A. and Ural'tseva, Nina N, Linear and quasilinear elliptic equations,
Academic Press, New York-London, 1968, xviii+495.

\bibitem{O88} M.~Otani, Existence and nonexistence of nontrivial solutions of some nonlinear degenerate elliptic equations, J. Funct. Anal., \textbf{76} (1988), 140--159.

\bibitem{Pa} P.~Pansu, {M\'etriques de Carnot--Carath\'eodory et quasiisom\'etries des espaces sym\'et\-ri\-qu\-es de rang un}, Ann.
Math., \textbf{129} (1989), 1--60.

\bibitem{PW61} L.~E.~Payne, H.~F.~Weinberger, Some isoperimetric inequalities for membrane frequencies
and torsional rigidity, J. Math. Anal. Appl., \textbf{2} (1961), 210--216.

\bibitem{PS51} G.~P\'olya and G.~Szeg\"o, Isoperimetric inequalities in mathematical physics, Annals of
mathematical studies \textbf{27}, Princeton University Press, Princeton, (1951).

\bibitem{R94} L.~Rayleigh, The Theory of Sound, London, (1894/1896).


\bibitem{Vodop} S.~K.~Vodop'yanov, Weighted Sobolev spaces and the boundary behavior of solutions of degenerate
hypoelliptic equations, Siberian Math. J., \textbf{36} (1995), 278--300.

\bibitem{Vodop1} S.~K.~Vodop'yanov, V.~M.~Chernikov, Sobolev spaces and hypoelliptic equations, Trudy Inst. Mat., \textbf{29} (1995), 7--62.




\bibitem{WNL} N.~Wei, P.~Niu, H.~Liu, Dirichlet Eigenvalue Ratios for the $p$-sub-Laplacian in the Carnot Group,
J. Part. Diff. Eq., \textbf{22} (2009), 1--10.

\bibitem{X90} Ch.~J.~Xu, Subelliptic variational problems, Bull. Soc. Math. France, \textbf{118} (1990), 147--169.

\end{thebibliography}
\end{document}